\documentclass[leqno,12pt]{article} 
\usepackage{amsmath, amssymb}
\usepackage{amsthm} 
\usepackage{braket}
\usepackage{mathtools}
\DeclarePairedDelimiter{\abs}{\lvert}{\rvert}
\usepackage{latexsym}

\usepackage{cases}

\makeatletter
    
    \@addtoreset{equation}{section}
\makeatother

\theoremstyle{plain} 
\newtheorem{theorem}{\indent\sc Theorem}[section] 
\newtheorem{lemma}[theorem]{\indent\sc Lemma}

\newtheorem{proposition}[theorem]{\indent\sc Proposition}

\theoremstyle{definition} 
\newtheorem{definition}[theorem]{\indent\sc Definition}
\newtheorem{remark}[theorem]{\indent\sc Remark}

\newcommand{\deldel}{\sqrt{-1}\partial \overline{\partial}}
\newcommand{\dbar}{\overline{\partial}}
\newcommand{\e}{\varepsilon}

\newcommand{\poly}{\bigtriangleup}

\makeatletter

\def\address#1#2{\begingroup
\noindent\parbox[t]{7.8cm}{
\small{\scshape\ignorespaces#1}\par\vskip1ex
\noindent\small{\itshape E-mail}
\/: #2\par\vskip4ex}\hfill
\endgroup}
\makeatother

\title{Generalized K\"{a}hler Einstein metrics and uniform stability for toric Fano manifolds}

\author{
\textsc{Satoshi Nakamura$^{*}$}
} 

\date{}

\begin{document}

\maketitle


\footnote{ 
$^{*}$Partly supported by Grant-in-Aid for JSPS Fellowships for Young Scientists, Number 17J02783.
}

\begin{abstract}
We give a complete criterion for the existence of generalized K\"{a}hler Einstein metrics on toric Fano manifolds
from view points of a uniform stability in a sense of GIT and the properness of a functional on the space of K\"{a}hler metrics.
\end{abstract}


\section{Introduction}
In his paper \cite{M}, Mabuchi extended the notion of K\"{a}hler Einstein metrics for Fano manifolds with non vanishing Futaki character.
In this paper, we call them generalized K\"{a}hler Einstein metrics.
Let $X$ be an $n$-dimensional Fano manifold and $\omega \in 2\pi c_1(X)$ be a K\"{a}hler metric.
We denote the Ricci form for $\omega$ by $\mathrm{Ric}(\omega)= -\deldel\log\det\omega^n$.
The Ricci potential $f_{\omega}$ for $\omega$ is the function satisfying
\[
\mathrm{Ric}(\omega)-\omega=\deldel f_{\omega} \quad\text{and}\quad \int_X e^{f_{\omega}}\omega^n =\int_X \omega^n.
\]
Then $\omega = \sqrt{-1}g_{i \bar{j}}dz^i\wedge d\bar{z}^j$ is called {\it generalized K\"{a}hler Einstein} 
if the complex gradient vector field of $1-e^{f_{\omega}}$ is holomorphic, that is,
\[
\dbar \Bigl(g^{i\bar{j}} \frac{\partial (1-e^{f_{\omega}})}{\partial\bar{z}^j} \frac{\partial}{\partial z^i} \Bigr)=0.
\]

Generalized K\"{a}hler Einstein metrics do not necessarily exist.
Mabuchi \cite{M} introduced an obstruction for the existence of this metric.
Let
\[
\tilde{\mathfrak{h}}_{\omega}:=\Set{f\in C^{\infty}_{\mathbb{R}}(X) 
| \dbar \Bigl(g^{i\bar{j}} \frac{\partial f}{\partial\bar{z}^j} \frac{\partial}{\partial z^i} \Bigr)=0 \quad\text{and}\quad \int_Xf\omega^n=0},
\]
and let $\mathrm{pr}: L^2(X, \omega) \to \tilde{\mathfrak{h}}_{\omega}$ be the $L^2$-orthogonal projection 
where $L^2(X, \omega)$ is the Hilbert space of real $L^2$-functions on $(X,\omega)$.
Then we can define the holomorphic invariant $\alpha_X$ as follows:
\[
\alpha_{X}:=\max_X \mathrm{pr}(1-e^{f_{\omega}}).
\]
Indeed $\alpha_X$ is independent of the choice of $\omega$.
Furthermore if $X$ admits generalized K\"{a}hler Einstein metrics 
then $\alpha_X<1$ must hold. See \cite{M} for more details.

For general Fano manifolds, necessary and sufficient conditions for the existence of generalized K\"{a}hler Einstein metics are not known.
For toric Fano manifolds, Yao introduced the notion of the {\it relative Ding stability} in a sense of geometric invariant theory.
(After the present paper was submitted on the arXiv, Yao renamed this stability the {\it uniform relative Ding stability} in his replaced paper \cite{Y} independently. His definition of the uniform relative Ding stability is equivalent to the author's one in the present paper.)
This stability condition is equivalent to  $\alpha_X<1$,
and implies the existence of generalized K\"{a}hler Einstein metrics (Yao called them {\it Mabuchi metrics}).
See \cite{Y} for more details.

In this paper, we introduce a notion of stability called the {\it uniform relative Ding stability} 
for toric Fano manifolds, and 
establish the complete criterion for the existence of generalized K\"{a}hler Einstein metrics from the view point of  the properness of a functional on the space of K\"{a}hler metrics.
These lines to attack the existence problem of generalized K\"{a}hler Einstein metrics are expected to extend for general Fano manifolds.

The uniform relative Ding stability can be regarded  as a generalization of the uniform Ding stability for K\"ahler Einstein metrics defined by Berman \cite{B}.
For more recent developments around the uniform stability and the Ding stability, see for instance
Hisamoto \cite{H} (for toric constant scalar curvature K\"ahler metrics), 
Boucksom-Hisamoto-Jonsson \cite{BHJ} (for constant scalar curvature K\"ahler metrics),
Dervan \cite{D} (for twisted constant scalar curvature K\"ahler metrics)
and Fujita \cite{F} (for the volume of K\"ahler Einstein Fano manifolds).

The following is our main theorem.

\begin{theorem}\label{main}
Let $X$ be a toric Fano manifold.
Then following conditions are all equivalent.
\begin{enumerate}
\item $X$ admits a unique toric invariant generalized K\"{a}hler Einstein metic.
\item $\alpha_X<1$.
\item $X$ is uniform relative Ding stable (Definiton \ref{uniform}).
\item The modified Ding functional is proper (Definition \ref{proper}).
\end{enumerate}
\end{theorem} 

The proof will be done by showing ${\it 2} \Rightarrow {\it 3} \Rightarrow {\it 4} \Rightarrow {\it 1}$.
Recall that ${\it 1} \Rightarrow {\it 2}$ is the result of Mabuchi \cite{M} for general Fano manifolds.
Recall also that ${\it 2} \Rightarrow {\it 1}$ is the result of Yao \cite{Y}.
Therefore the author's main contribution is discussion in conditions ${\it 3}$ and ${\it 4}$.  



{\sc Acknowledgements:}
The author would like to thank Professor Shigetoshi Bando, Professor Shunsuke Saito and Doctor Ryosuke Takahashi for several helpful comments and constant encouragement.
He also would like to thank Professor Mattias Jonsson for introduction to related works around the uniform stability.

\section{Generalized K\"{a}hler Einstein metrics on toric Fano manifolds}\label{GKE}
Let $\poly$ be an open polytope  in $\mathbb{R}^n$ such that the closure $\overline{\poly}$ is a reflexive integral Delzant polytope,  
and $X_{\poly}$ be the corresponding toric Fano manifold with the open dense $(\mathbb{C}^*)^n$-action.
Note that $0\in\mathbb{R}^n$ is only integral point in $\poly$.
In the following, we will just write $X$ as $X_{\poly}$ for simplicity.
Let $(\mathbb{C}^*)^n = (S^1)^n \times \mathbb{R}^n$ be the standard decompositoin,
and let $\xi_i := \log |z_i|^2$ be the coordinate of $\mathbb{R}^n$ where $\{z_i\}$ is the standard coordinate of $(\mathbb{C}^*)^n$.
The following lemma is well-known.

\begin{lemma}
Let $\omega$ be a $(S^1)^n$-invariant K\"{a}hler metirc on X.
Then $\omega$ is determined by a smooth convex function (called K\"{a}hler potential) $\phi = \phi(\xi_1, \dots, \xi_n)$ on $\mathbb{R}^n$, that is,
$\omega = \deldel \phi$ on $(\mathbb{C}^*)^n$.
Moreover its gradient $\nabla\phi$ gives a diffeomorphism from $\mathbb{R}^n$ to $\poly$. 
\end{lemma}
Let $\omega \in 2\pi c_1(X)$ be a $(S^1)^n$-invariant reference K\"{a}hler metirc 
and $\phi_0$ is the K\"{a}hler potencial of $\omega$.
Let $u_0$ be the Legendre dual of $\phi_0$, that is, 
\begin{equation*}
u_0(x) = \xi\cdot x - \phi_0(\xi) \quad\text{and}\quad x=\nabla\phi_0(\xi).
\end{equation*}
Then $u_0$ is a smooth convex function (called {\it symplectic potencial}) on $\poly$.
Let
\[
\mathcal{C}:= \Set{u \in C^0(\overline{\poly}) | u \text{ is convex on } \poly \text{ and } u-u_0 \in C^{\infty}(\overline{\poly})}. 
\]
Abreu \cite{Ab} showed that there is the bijection between $\mathcal{C}$ and the space of $(S^1)^n$-invariant K\"{a}hler metrics in $[\omega]$ through the Legendre duality.

Let us consider the {\it modified Ding functional} \cite{Y}
\[
\mathcal{D}(u) = -\log\int_{\mathbb{R}^n} e^{-(\phi-\inf\phi)}d\xi - u(0) + \int_{\poly}u\cdot l dx
\quad\text{for } u\in\mathcal{C},
\]
where $\phi$ is the Legendre dual of $u$ and $l$ is the unique affine linear function on $\poly$ such that
\[
- u(0) + \int_{\poly}u\cdot l dx = 0 \quad\text{ for any affine linear } u. 
\]
The critical point of the modified Ding functional is the generalized K\"{a}hler Einstein metric.
Indeed the derivative of $\mathcal{D}$ at $\phi$ is 
(Note that $\delta \phi = - \delta u(\nabla\phi)$ and $u(0)=-\inf\phi$.)
\begin{equation}\label{derivative}
\delta\mathcal{D} 
= \int_{\mathbb{R}^n} \delta\phi \Bigl( \frac{e^{-\phi}}{\int_{\mathbb{R}^n}e^{-\phi}} 
- l(\nabla\phi)\det(\nabla^2\phi) \Bigr)d\xi,
\end{equation}
and we can see that a smooth solution $\phi$ of 
\begin{equation}\label{eq}
\frac{e^{-\phi}}{\int_{\mathbb{R}^n}e^{-\phi}} = l(\nabla\phi)\det(\nabla^2\phi)
\end{equation}
such that its gradient $\nabla\phi$ gives a diffeomorphism from $\mathbb{R}^n$ to $\poly$
defines the generalized K\"{a}hler Einstein metric (cf. \cite[Theorem 14]{Y}).

\section{Uniform relative Ding stability}
Notation of this section is same as previous sections.
Let us consider the {\it relative Ding-Futaki invariant} \cite{Y}
\[
\mathcal{I}(u):= -u(0) + \int_{\poly} u\cdot l dx \quad\text{for } u \in \mathcal{C}.
\]
This is nothing but the linear term of the modified Ding functional.
We introduce the uniform relative Ding stability as a uniformly lower bound estimate of $\mathcal{I}$.
\begin{definition}\label{uniform}
A toric Fano manifold $X$ is {\it uniform relative Ding stable} if there is a constant $\lambda>0$ such that
\begin{equation}\label{bound}
\mathcal{I}(u) \geq \lambda \int_{\poly} u dx
\end{equation}
holds for every {\it normalized} convex function $u \in \tilde{\mathcal{C}}:= \Set{u\in\mathcal{C} | u\geq u(0)=0 }$.
\end{definition}
Although we can also define another stability by the lower bound estimate by $\lambda\int_{\partial\poly}ud\sigma$ as the analog of the uniform K-stability defined by Donaldson \cite{D}, 
this condition is stronger than that of Definition \ref{uniform}.
Indeed there is the uniform estimate as follows \cite{D}:
there is a constant $C>0$ such that
$
\int_{\poly}udx \leq C \int_{\partial\poly}ud\sigma
$
holds for all $u\in\tilde{\mathcal{C}}$.
Our uniform stability follows from the condition $\alpha_X<1$ immediately.
\begin{lemma}\label{A}
Let $X$ be a toric Fano manifold.
If the condition $\alpha_X<1$ holds then $X$ is uniform relative Ding stable.
\end{lemma}
\begin{proof}
As Yao pointed out in \cite{Y}, 
the invariant $\alpha_X$ is given explicitly by
\[
\alpha_X = \max_{\overline{\poly}}\{1- \abs\poly l\},
\]
where $\abs\poly$ is the volume of $\poly$.
Therefore for any $u\in\tilde{\mathcal{C}}$, we have
\[
\mathcal{I}(u) = \int_{\poly}u\cdot ldx \geq \frac{1-\alpha_X}{\abs\poly}\int_{\poly}udx.
\]
\end{proof}

\section{Properness of the modified Ding functional}
First we recall a functional on the space of K\"{a}hler metrics.
Let $(X, \omega)$ be a $n$-dimensional compact K\"{a}hler manifold.
Let $G$ be any maximal compact subgroup of $\mathrm{Aut}(X)$.
If $\omega$ is $G$-invariant then we define
\[
\mathcal{M}_G(\omega) 
=\Set{\phi\in C^{\infty}(X) | \omega_{\phi}:= \omega +\deldel\phi >0 \text{ and } \phi \text{ is } G \text{ -invariant}.}.
\]
For any $\phi\in\mathcal{M}_G(\omega)$, we let
\[
J(\phi)= \int_0^1 \int_X \dot{\phi_t}(\omega^n-\omega^n_{\phi_t})\frac{1}{n!}\wedge dt,
\]
where $\{\phi_t\}_{t\in[0,1]}$ is a path in $\mathcal{M}_G(\omega)$ connecting $0$ to $\omega$.

In general, the properness of a functional on $\mathcal{M}_G(\omega)$ 
is defined as a uniform lower bound estimate by the functional $J$ \cite[Definition 6.8]{T}.
In the case of previous sections where $X$ is toric Fano manifold and $\omega\in2\pi c_1(X)$ is $(S^1)^n$-invariant reference K\"{a}hler metirc,
there exists a uniform constant $C>0$ such that
\begin{equation}\label{J}
\Bigl| J(\phi)-\int_{\poly}udx \Bigr| \leq C
\end{equation}
holds for all $u\in\tilde{\mathcal{C}}$ \cite[Lemma 2.2]{ZZ},
where $\phi$ is the $(S^1)^n$-invariant function associated from the K\"{a}hler metric defined by $u$ .
Thus we can define the properness of the modified Ding functional for toric Fano manifolds as follows:
\begin{definition}\label{proper}(cf. \cite[Definition 6.8]{T})
In the same notation as in the section \ref{GKE},
the modified Ding functional $\mathcal{D}$ is {\it proper} 
if there exists an increasing function $\mu(r)$ on $\mathbb{R}$ with the property
\[
\lim_{r\to\infty}\mu(r) = \infty,
\]
such that
\begin{equation}\label{inequality}
\mathcal{D}(u) \geq \mu \Bigl( \int_{\poly}udx \Bigr)
\end{equation}
holds for all $u\in\tilde{\mathcal{C}}$.
\end{definition}
Then we have the following:

\begin{proposition}\label{B}
Let $X$ be a toric Fano manifold.
Suppose $X$ is uniform relative Ding stable, 
that is, there exists a constant $\lambda>0$ satisfying
$
\mathcal{I}(u) \geq \lambda \int_{\poly} u dx
$
for all  $u \in \tilde{\mathcal{C}}$.
Then there exists $\delta >0$ depending only on $\lambda$ such that
\[
\mathcal{D}(u) \geq \delta\int_{\poly}udx -C_{\delta}
\]
holds for all $u\in\tilde{\mathcal{C}}$.
In particular the modified Ding functional is proper.
\end{proposition}
\begin{proof}
For fixed $v_0 \in \mathcal{C}$ and its Legendre dual $\psi_0$, 
we define a smooth function $A$ 
by
\[
A(\nabla\psi_0) = \frac{e^{-\psi_0}}{\int_{\mathbb{R}^n}e^{-\psi_0}} \det(\nabla^2\psi_0)^{-1}.
\]
Then $v_0$ minimizes the following functional on $\mathcal{C}$:
\[
\mathcal{D}_A(u):= -\log\int_{\mathbb{R}^n} e^{-(\phi-\inf\phi)}d\xi - u(0) + \int_{\poly}u\cdot A dx
\]
where $\phi$ is the Legendre dual of $u$.
Indeed $v_0$ is the critical point of $\mathcal{D}_A$ (cf. \eqref{derivative}),
and the nonlinear term of $\mathcal{D}_A$ is convex with respect to $\phi$ by the Prekopa-Leindler inequality.
Let $-C_0:=\mathcal{D}_A(v_0)$ and let
\[
\mathcal{I}_A(u) := - u(0) + \int_{\poly}u\cdot A dx.
\]
Now we compute the difference between $\mathcal{I}$ and $\mathcal{I}_A$.
Taking a $\delta>0$ and for any $u\in\tilde{\mathcal{C}}$, we have 
\begin{eqnarray*}
|\mathcal{I}(u)-\mathcal{I}_A(u)| &\leq& C \int_{\poly}udx\\
                                                   &=& C \Bigl( (1+\delta)\int_{\poly}udx-\delta\int_{\poly}udx \Bigr)\\
                                                   &\leq& C_{\delta, \lambda}\mathcal{I}(u) - C\delta\int_{\poly}udx.
\end{eqnarray*}
The last inequality follows from the assumption of the uniform relative Ding stability.
It follows that
\[
(1+C_{\delta, \lambda})\mathcal{I}(u)\geq\mathcal{I}_A(u) +C\delta\int_{\poly}udx.
\]
Therefore we get 
\begin{eqnarray*}
\mathcal{D}(u) &=& -\log\int_{\mathbb{R}^n} e^{-(\phi-\inf\phi)}d\xi + \mathcal{I}(u)\\
                       &\geq& -\log\int_{\mathbb{R}^n} e^{-(\phi-\inf\phi)}d\xi
                                   + \mathcal{I}_A\Bigr(\frac{u}{1+C_{\delta, \lambda}}\Bigr) 
                                   + \frac{C\delta}{1+C_{\delta, \lambda}}\int_{\poly}udx\\
                      &\geq&  - \log\int_{\mathbb{R}^n} \exp\Bigl(-\frac{\phi-\inf\phi}{1+C_{\delta, \lambda}}\Bigr)d\xi
                                   + \mathcal{I}_A\Bigl(\frac{u}{1+C_{\delta, \lambda}}\Bigr) 
                                   + \frac{C\delta}{1+C_{\delta, \lambda}}\int_{\poly}udx\\
                      &=&          \mathcal{D}_A\Bigr(\frac{u}{1+C_{\delta, \lambda}}\Bigr) -n\log(1+C_{\delta, \lambda})
                                   + \frac{C\delta}{1+C_{\delta, \lambda}}\int_{\poly}udx\\
                      &\geq& -C_0 -n\log(1+C_{\delta, \lambda}) + \frac{C\delta}{1+C_{\delta, \lambda}}\int_{\poly}udx.
\end{eqnarray*}
Replacing $ \frac{C\delta}{1+C_{\delta, \lambda}}$ by $\delta$, this completes the proof.
\end{proof}

\section{Proof of Theorem \ref{main}}
By Lemma \ref{A} and Proposition \ref{B}, 
it remains to show that the properness of the modified Ding functional implies the solvability of the equation \eqref{eq}.
Then we have the following.
\begin{proposition}\label{positive}
Under the assumption of the properness of the modified Ding functional, we have
\[
l>0 \quad\text{on}\quad \overline{\poly}.
\]
 \end{proposition}
 
\begin{proof} 
Since $l$ is affine linear,  it suffice to show that $l(p)>0$ for any vertex $p$ of $\overline{\poly}$.
Let $\{v_i\}_{i\gg 1}$ be the sequence of smooth convex functions on $\overline{\poly}$ satisfying the followings:
\begin{itemize}
\item $v_i \geq v_i(0) =0$.
\item $v_i$ tends to the $K$ $(>0)$ times of the Dirac function for $p$ as $i \to \infty$.
\end{itemize}
A construction of $v_i$ is written in the end of this proof.

Then the convex function $u_i := \tilde{u}_0 + v_i$ is in  $\tilde{\mathcal{C}}$, 
where $\tilde{u}_0 \in \tilde{\mathcal{C}}$ is the normalization of the symplectic potencial $u_0 \in \mathcal{C}$ of the $(S^1)^n$-invariant reference K\"ahler metric on $X$.
Let $\phi_i$ be the Legendre dual of $u_i$, that is, $\phi_i(\xi)=\sup_{x\in\poly}(x\cdot\xi-u_i(x))$.
Note that
\[
\inf_{\mathbb{R}^n}\phi_i=0 \quad\text{and}\quad \phi_i(\xi)\leq\sup_{x\in\poly}(x\cdot\xi)
\]
since $\inf_{\mathbb{R}^n}\phi_i = -u_i(0)$ by a property of the Legendre duality, and $u_i \geq u_i(0) = 0$ by the definition of $u_i$.
It follows that
\[
\log\int_{\mathbb{R}^n}e^{-(\phi_i-\inf\phi_i)} d\xi \geq \log\int_{\mathbb{R}^n}e^{-\sup_{x\in\poly}(x\cdot\xi)}d\xi.
\] 
By the properness of the modified Ding functional \eqref{inequality}, we thus have
\[
\int_{\poly}u_i\cdot l dx \geq \mu\Bigl(\int_{\poly}u_idx\Bigr) + \log\int_{\mathbb{R}^n}e^{-\sup_{x\in\poly}(x\cdot\xi)}d\xi.
\]
By taking $i\to\infty$, we have
\[
K\cdot l(p)+\int_{\poly}l\cdot \tilde{u}_0 dx \geq \mu\Bigl(K+\int_{\poly}\tilde{u}_0 dx\Bigr)
+ \log\int_{\mathbb{R}^n}e^{-\sup_{x\in\poly}(x\cdot\xi)}d\xi.
\]
Namely,
$
K\cdot l(p) \geq \mu(K+C_1)-C_2
$
holds for some constants $C_i$ independent of $K$.
Hence, by taking $K$ sufficiently large, we get $l(p)>0$.

Finally, for a construction of $v_i$,
we take an affine linear function $w_i$ on $\mathbb{R}^n$ satisfying the followings:
\begin{enumerate}
\item $w_i(p)=Ki.$
\item $\int_{\overline{\poly}\cap \{w_i \geq 0\}} w_i dx = K.$
\item There exists $r_i>0$ such that $\lim_i r_i =0$ 
and $\Set{x\in\overline{\poly} | w_i(x)\geq 0} \subset \Set{x\in\overline{\poly} | \|x-p\|_{\text{Euc}}< r_i}$.
\end{enumerate}
Let $\hat{w}_i := \max_{\mathbb{R}^n}\{0, w_i\}$.
Note that $\hat{w}_i |_{\overline{\poly}}$ tends to the $K$ times of the Dirac function for $p$ as $i\to\infty$. 
To take a smoothing of $\hat{w}_i |_{\overline{\poly}}$, 
let us consider the convolution $\hat{w}_i \star \rho_{\e}$, where $\rho_{\e} \geq 0$ 
is the smooth mollifier on $\mathbb{R}^n$ whose support is in $B_{\e}(0)$. 
For sufficiently large $i$ and small $\e >0$, it is easy to see that $\hat{w}_i \star \rho_{\e}$ is smooth and convex on $\overline{\poly}$, 
and satisfies $\hat{w}_i \star \rho_{\e} \geq \hat{w}_i \star \rho_{\e}(0) = 0$ on $\overline{\poly}$. 
Note that $\hat{w}_i \star \rho_{\e}  |_{\overline{\poly}}$ uniformly converges to $\hat{w}_i  |_{\overline{\poly}}$ as $\e\to 0$.
Thus we can take $\e_i >0$ such that 
\[
\| \hat{w}_i \star \rho_{\e_i} - \hat{w}_i \|_{L^{\infty}(\overline{\poly})} \leq \frac{1}{i}.
\]
Then we define $v_i$ as $\hat{w}_i \star \rho_{\e_i}  |_{\overline{\poly}}$.
\end{proof}

The condition $l>0$ guarantees the non-degeneracy of the equation \eqref{eq}.
Hence the solvability of the equation of \eqref{eq} follows immediately from the result for real Monge-Amp\`{e}re equations by Berman and Berndtsson \cite[Theorem 1.1]{BB}.
This completes the proof of Theorem \ref{main}. 

\begin{remark}
By the same argument as in Proposition \ref{positive}, 
we can prove directly that the uniform relative Ding stability implies the condition $l>0$ on $\overline{\poly}$ (See also \cite[Proposition 12]{Y}).
Indeed, we apply the inequality \eqref{bound} of the uniform relative Ding stability to the function $u_i = \tilde{u}_0 + v_i \in \tilde{\mathcal{C}}$ in Proposition \ref{positive}.
By taking $i\to\infty$, we then have
\[
K\cdot l(p) + \int_{\poly}\tilde{u}_0 \cdot l dx \geq \lambda \cdot K + \lambda \int_{\poly} \tilde{u}_0 dx.
\]
Hence, by taking $K$ sufficiently large, we get $l(p)>0$.
\end{remark}

\bigskip

\address{
Mathematical Institute \\ 
Tohoku University \\
Sendai 980-8578 \\
Japan
}
{satoshi.nakamura.r8@dc.tohoku.ac.jp}
\end{document}